\documentclass{amsart}
\usepackage[utf8]{inputenc}
\usepackage[english]{babel}
\usepackage{verbatim}
\usepackage{graphicx}
\usepackage{amsmath,amsthm,amssymb}
\usepackage{mathrsfs}
\usepackage{afterpage} 
\usepackage{tcolorbox} 
\usepackage{float}
\usepackage{subfig}
\usepackage{enumerate}
\usepackage{todonotes}
\usepackage{tikz-cd}
\usepackage{faktor}

\newtheorem{Theorem}{Theorem}
\newtheorem{Corollary}[Theorem]{Corollary}
\newtheorem{Proposition}[Theorem]{Proposition}
\newtheorem{Lemma}[Theorem]{Lemma}
\newtheorem{Definition}[Theorem]{Definition}
\newtheorem{Remark}[Theorem]{Remark}
\newtheorem{Property}[Theorem]{Property}

\newtheorem{Example}[Theorem]{Example}

\usepackage{xcolor,cancel}

\begin{document}

\title{$G$-strong subdifferentiability and applications to norm attaining subspaces}

\author[J. Falc\'{o}]{Javier Falc\'{o}}
\address[Javier Falc\'{o}]{Departamento de An\'{a}lisis Matem\'{a}tico,
	Universidad de Valencia, Doctor Moliner 50, 46100 Burjasot (Valencia), Spain} \email{francisco.j.falco@uv.es}
\author[D. Isert]{Daniel Isert}
\address[Daniel Isert]{Departamento de An\'{a}lisis Matem\'{a}tico,
	Universidad de Valencia, Doctor Moliner 50, 46100 Burjasot (Valencia), Spain}
\email{daniel.isert@uv.es}

\thanks{The first author was supported by grant PID2021-122126NB-C33 funded by MCIN/AEI/10.13039/501100011033 and by “ERDF A way of making Europe”.}


\subjclass[2020]{}

\begin{abstract}
We study the reflexivity and strong subdifferentiability within the framework of group invariant mappings. We show that a Banach space is $G$-reflexive if the norm of its dual is $G$-strong subdifferentiable. To do this, we extend numerous classical concepts in functional analysis such as weak and weak-star topologies, the polar of a set, duality mapping, to the framework of group invariant mappings. We also extend many classical results in functional analysis including Banach-Alaoglu-Bourbaki's theorem, James' theorem, Moreau's maximum formula, and Krein-Smulian's theorem, to this context. To conclude, we provide an application of these new results by providing sufficient conditions to ensure the existence of closed Banach spaces inside the set of norm-attaining functionals of a Banach space.
\end{abstract}

\maketitle


\section{Introduction}

In this paper, we develop the theory of group-invariant mappings—mappings that remain unchanged under the action of a topological group—within the broader context of infinite-dimensional analysis. Our primary focus is to explore the geometric properties of Banach spaces and examine the extent to which key results in their geometry hold when the involved mappings are group-invariant.

 Let us start by introducing the notation and key definitions that will be used throughout the paper. Let $X$ and $Y$ denote Banach spaces over $\mathbb{K}$, the field of real or complex numbers. We denote by $X^*$, $B_X$, and $S_X$  the dual space, the unit ball, and the unit sphere of the Banach space $X$, respectively. We reserve $G$ to denote a topological group of invertible bounded linear mappings, equipped with the relative topology from $\mathcal{L}(X)$, and $g$ for elements of the group $G$. We define  the corresponding dual and bidual group as 
\[
G^{*} = \left\{g^{*} \in \mathcal{L}(X^{*}) ~ | ~ g \in G\right\} \hbox{ and } G^{**} = \left\{g^{**} \in \mathcal{L}(X^{**}) ~ | ~ g \in G\right\}.
\]
The concept of $G$-invariance is central to our analysis. There are three key notions of $G$-invariance, defined as follows:
\begin{enumerate}
\item A point $x \in X$ is $G$-invariant, or invariant under the action of $G$ if $g(x) = x$ for every $g \in G$.

\item A set $K \subset X$ is $G$-invariant if for every $g \in G$, $g(K) = K$.

\item A mapping $f \colon X \to Y$ is $G$-invariant if $f(g(x)) = f(x)$ for every $x \in X$ and every $g \in G$
\end{enumerate}

Throughout the paper, we assume that the elements of $G$ are isometric isomorphisms on $X$, meaning the norm of the Banach space $(X, \Vert \cdot \Vert)$ is invariant under the action of $G$. Under this assumption, we will denote by $X_G$ the set of $G$-invariant points of $X$, $X_G^*$ will be the set of $G$-invariant linear and continuous functionals on $X$.

Lastly, we define the concept of a $G$-symmetrization point. Given a Banach space $X$ and a compact topological group $G$ acting on $X$, the $G$-symmetrization point of $x \in X$ is defined as
\[
\overline{x} = \int_{G}g(x)d\mu(g),
\] 
where $\mu$ is the Haar measure and the integral is the Bochner integral.

Finally, recall that a function $f \colon X \to \mathbb{R}$ is convex with respect to the group $G$, or simply $G$-convex, given that
\[
f(\overline{x}) \leq \int_{G}f(g(x))d\mu(g) \quad \forall x \in X.
\]
And a function $f \colon X \to \mathbb{R}$ is linear with respect to the group $G$, or simply $G$-linear, if
\[
f(\overline{x}) = \int_{G}f(g(x))d\mu(g) \quad \forall x \in X.
\]
For more on these definitions see \cite{FaIs}.

The structure of the paper is as follows. In Section \ref{Seccio 2}, we provide the necessary background definitions for understanding Sections \ref{Seccio 3} and \ref{Seccio 4}. This section also includes a $G$-invariant version of the Banach-Alaoglu-Bourbaki theorem and Kakutani's theorem, as well as an in-depth study of the relationships between the spaces $X$, $X^*$, and $X^{**}$, and their $G$-invariant counterparts, $X_{G}$, $X^{*}_G$, and $X^{**}_{G^*}$. Some of these results are presented without proof since they can easily be obtained from the new definitions. However, we include them for the sake of completeness. The section concludes with the derivation of a commutative diagram.

In Section \ref{Seccio 3}, building on the previous work, we focus on several fundamental theorems in functional analysis, including James' theorem, Moreau's maximum formula, and Krein-Smulian's theorem. Finally, Section \ref{Seccio 4} offers a study of $G$-dissipative operators and culminates in the paper's main result: a Banach space $X$ is $G$-reflexive provided $X^*$ is $G$-strongly subdifferentiable. 

 To conclude section \ref{Seccio 4}, as an application of the framework developed here, we address a problem posed by Gilles Godefroy in 2001, concerning the lineability of the set of norm-attaining functionals on Banach spaces. As usual, we denote by \( \text{NA}(X) \) the set of functionals in \( X^* \) that attain their norm on \( X \), meaning there exists some \( x \in X \) with \( \|x\| = 1 \) such that \( \|f\| = f(x) \). The problem posed by Godefroy asks whether, for every infinite-dimensional Banach space \( X \), the set \( \text{NA}(X) \) contains a two-dimensional linear subspace.

 Martin Rmoutil provided a negative solution to this problem by giving an example of space whose set of norm-attaining is “extremely non-lineable” in the sense that the set \( \text{NA}(\mathbb{R}) \) does not contain any two-dimensional subspace. We recommend \cite{KadetsLopezMartin, Martin,Rmoutil} for more details in this topic.

To continue with the study of structural properties of $NA(X)$ we will provide in Section \ref{Seccio 4} sufficient conditions to ensure that the set $NA(X)$ contains a vectorial space. Furthermore, we will also highlight the relation between the existence of group invariant functionals that attain their norm and the existence of vectorial spaces of norm attaining functionals.

\section{Preliminar definitions and properties}\label{Seccio 2}

We start by presenting the definition of the $G$-weak and $G$-weak-star topologies.
\begin{Definition}
    Let $X$ be a normed space and $G$ a compact topological group acting on $X$, the weak group invariant topology on $X$ is generated by a basis consisting of sets of the form
    \[
    \left\{x \in X ~~ | ~~ \langle f_{i}, x - x_{0} \rangle < \epsilon, \, \hbox{ for } 1 \leq i \leq n\right\},
    \]
    for all choices of $x_{0} \in X$, $f_{1}, \dots, f_{n} \in X_{G}^{*}$ and $\epsilon > 0$. We denote this topology by $w_{G}$ or $\sigma_{G}(X, X^{*})$.
\end{Definition}

\begin{Definition}
    Let $X$ be a normed space and $G$ a compact topological group acting on $X$, the weak-star group invariant topology on $X^{*}$ is generated by a basis consisting of sets of the form
    \[
    \left\{f \in X^{*} ~~ | ~~ \langle f - f_{0}, x_{i} \rangle < \epsilon, \, \hbox{ for } 1 \leq i \leq n\right\},
    \]
    for all choices of $f_{0} \in X^{*}$, $x_{1}, \dots, x_{n} \in X_{G}$ and $\epsilon > 0$. We denote this topology by $w^{*}_{G}$ or $\sigma_{G}(X^{*}, X)$.
\end{Definition}

Notice the following relation between the $w$ topology and the $w_{G}$ topology.
\begin{Proposition}
\label{weaker-topologies}
  Let $X$ be a normed space and $G$ a compact topological group acting on $X$. If $X^{*}_{G} \subsetneq X^{*}$, then the weak group invariant topology on $X$, $w_{G}$, is strictly weaker than the weak topology of $X$, $w$.
\end{Proposition}
\begin{proof}
    If $X^{*}_{G} \subsetneq X^{*}$, then, there exists $f \in X^{*} \backslash X_{G}^{*}$, $x_{0}\in X$ and $g \in G$ such that $\langle f, x_{0}\rangle \neq \langle f, g(x_{0}) \rangle$. Let $\epsilon < |\langle f, g(x_{0}) \rangle - \langle f, x_{0} \rangle|$ and consider
    \[
    U = \left\{x \in X ~~ | ~~ \langle f, x\rangle - \langle f, x_{0} \rangle < \epsilon\right\},
    \]
    we claim that $U \in w \backslash w_{G}$. By definition $U \in w$. We proceed by contradiction to show that $U \notin w_{G}$. If $U \in w_{G}$, then $U = \cup_{\alpha}U_{\alpha}$ where
    \[
    U_{\alpha} = \left\{x \in X ~~ | ~~ \langle f_{i, \alpha}, x - y_{\alpha} \rangle < \epsilon_{\alpha}, \hbox{ for } 1 \leq i \leq n\right\},
    \]
    for all choices of $y_{\alpha} \in X$ and $f_{i, \alpha} \in X^{*}_{G}$. But observe now that, $x_{0} \in U$, and since $U \in w_{G}$, $g(x_{0}) \in U$. However $|\langle f, g(x_{0}) \rangle - \langle f, x_{0}\rangle| > \epsilon$, a contradiction.
\end{proof}

We first observe this result which is quite different to what happens with the topologies in the non group invariant version.
\begin{Property}
    Let $X$ be a Banach space and $G$ a compact topological group acting on $X$. If $X^{*}_{G} \subsetneq X^{*}$ then the topologies $w_{G}$ and $w^{*}_{G}$ are not Hausdorff.
\end{Property}
\begin{proof}
    We are only going to show the result for the $w_{G}$ topology. The $w^{*}_{G}$ case follows analogously. Pick $x \in X$ which is not $G$-invariant. Then, $x \neq g(x)$ for some $g \in G$. By definition, any open set containing $x$ contains an element
    \[
    U = \left\{y \in X ~~ | ~~ \langle f_{i}, y - x \rangle < \epsilon, \, \hbox{ for } 1 \leq i \leq n\right\}
    \]
    for some $f_{1}, \dots, f_{n} \in X^{*}_{G}$ and $\epsilon>0$. Now, any open set containing $g(x)$, again, by definition, contains a set 
    \[
    V = \left\{y \in X ~~ | ~~ \langle h_{i}, y - g(x) \rangle < \epsilon', \, \hbox{ for } 1 \leq i \leq m\right\}
    \]  for some $h_{1}, \dots, h_{m} \in X^{*}_{G}$ and $\epsilon'>0$.
    Observe that
    $x\in U \cap V\ne\emptyset$, by the $G$-invariance of $h_1,\ldots, h_m$. Thus, we cannot find two disjoint neighborhoods of $x$ and $g(x)$.
\end{proof}

\begin{Proposition}
    Let $X$ be a normed space and $G$ a compact topological group acting on $X$.
    \begin{enumerate}[(i)]
        \item Let $\left\{f_{n}\right\}_{n} \subseteq X_{G}^{*}$, $f \in X_{G}^{*}$. Then $f_{n} \xrightarrow{w_{G}^{*}} f$  if, and only if, $\lim_{n \to +\infty}\langle f_{n}, x \rangle = \langle f, x \rangle$ for every $x \in X$.

        \item Let $\left\{x_{n}\right\}_{n} \subseteq X$, $x \in X$. Then $x_{n} \xrightarrow{w_G} g(x)$  for every $g \in G$ if, and only if, $\lim_{n \to +\infty}\langle f, x_{n} \rangle = \langle f, x \rangle$ for every $f \in X_G^{*}$.
    \end{enumerate}
\end{Proposition}

\begin{Property}
    Let $X$ be a Banach space and $G$ be a compact topological group acting on $X$. Choose a sequence $\left\{f_{n}\right\}_{n \in \mathbb{N}} \subseteq X^{*}_{G}$ and a function $f \in X^{*}_{G}$. If $\langle f_{n}, x \rangle \to \langle f, x \rangle$ for every $x \in X_{G}$, then $\langle f_{n}, x \rangle \to \langle f, x \rangle$ for all $x \in X$. Obviously the converse is also true.
\end{Property}

The proof of this result follows directly from the fact that $<f,x> =<f, \overline{x}>$ for every $x\in X$ and $f\in X_G^*$.

Our next result is the group invariant version of the classical Banach-Alaoglu-Bourbaki. The proof of this result follows from the Banach-Alaoglu-Bourbaki theorem and Proposition \ref{weaker-topologies}.
\begin{Theorem}[Group invariant Banach-Alaoglu-Bourbaki's theorem]\label{Banach-Alaoglu-Bourbaki}
    Let $X$ be a Banach space and $G$ a compact topological group acting on $X$. Then $B_{X^{*}} $ is $w_{G}^{*}$-compact.
\end{Theorem}

To continue, let us introduce the notion of $G$-reflexive space.
\begin{Definition}\label{Injeccio canonica G-inv}
    Let $X$ be a Banach space and $G$ a compact topological group acting on $X$. We say that $X$ is $G$-reflexive if the canonical injection $\pi \colon X \to X^{**}$ is $G$-surjective, i.e., $\pi(X_{G}) = X_{G^{**}}^{**}$.
\end{Definition}

Before we continue, let us recall the following lemma, it can be found in \cite[Lemma 3.3]{Brezis}. This result will be needed in the proof of the Goldstine's theorem.
\begin{Lemma}[Helly]\label{Helly's lemma}
    Let $X$ be a Banach space. Let $f_{1}, \dots, f_{k}$ be given in $X^{*}$, and let $\gamma_{1}, \dots, \gamma_{k}$ be given in $\mathbb{R}$.The following properties are equivalent:
    \begin{enumerate}
        \item $\forall \epsilon > 0$, $\exists x_{\epsilon} \in B_{X}$ such that
        \[
        |\langle f_{i}, x_{\epsilon} \rangle - \gamma_{i}| < \epsilon \quad \forall \, 1 \leq i \leq k.
        \]

        \item $\left|\sum_{i=1}^{k}\beta_{i}\gamma_{i}\right| \leq \left\Vert \sum_{i=1}^{k}\beta_{i}f_{i} \right\Vert$ for all $\beta_{1}, \dots, \beta_{k} \in \mathbb{R}$.
    \end{enumerate}
\end{Lemma}

\begin{Theorem}[Group invariant Goldstine's theorem]\label{Goldstine}
    Let $X$ be a Banach space and $G$ a compact topological group acting on $X$. Then $\overline{B_{X}}^{w_{G}^{*}} = B_{X^{**}_{G^{**}}}$.
\end{Theorem}
\begin{proof}
    Let $\xi \in B_{X^{**}_{G^{**}}}$ and $V$ be a neighborhood of $\xi$ for the $w_{G}^{*}$ topology. We want to see that $V \cap \pi(B_{X}) \neq \emptyset$. Notice that, for some given $f_{1}, \dots, f_{k} \in X^{*}_{G}$, and $\epsilon > 0$:
    \[
    V' = \left\{ \eta \in X^{**} ~~ | ~~ |\langle \eta - \xi, f_{i} \rangle| < \epsilon, \hspace{0.2cm} \forall \, 1 \leq i \leq k, \hspace{0.2cm} f_i\in X_G^* \right\}\subseteq V.
    \]
    Therefore, it is enough to find an $x \in B_{X}$ such that $\pi(x) \in V'$, i.e.,
    \[
    |\langle f_{i}, x \rangle - \langle \xi, f_{i} \rangle| < \epsilon \quad \text{ for all } \, 1 \leq i \leq k.
    \]
    Define $\gamma_{i} = \langle \xi, f_{i} \rangle$, by Lemma \ref{Helly's lemma} it suffices to check that
    \[
    \left|\sum_{i=1}^{k}\beta_{i}\gamma_{i}\right| \leq \left\Vert \sum_{i=1}^{k}\beta_{i}f_{i} \right\Vert.
    \]
    But this is clear since:
    \[
    \left|\sum_{i=1}^{k}\beta_{i}\gamma_{i}\right| = \left|\sum_{i=1}^{k}\beta_{i}\langle \xi, f_{i} \rangle\right| = \left|\left\langle \xi, \sum_{i=1}^{k}\beta_{i}f_{i} \right\rangle\right| \leq \left\Vert \sum_{i=1}^{k}\beta_{i}f_{i} \right\Vert,
    \]
    where we have used that $\Vert \xi \Vert \leq 1$.
\end{proof}

\begin{Theorem}[Group invariant Kakutani's theorem]
    \label{Kakutani}
    Let $X$ be a Banach space and $G$ a compact topological group acting on $X$. Then $X$ is $G$-reflexive if, and only if, $B_{X}$ is compact in the $\sigma_{G}(X, X^{*})$ topology.
\end{Theorem}
\begin{proof}
    To prove the if statement, assume $X$ is $G$-reflexive. Then, since $w_{G} \subseteq w$ and we know that $B_{X}$ is compact in the $\sigma(X^{**}, X^{*})$ topology, we have that $B_{X}$ is compact in the $\sigma_{G}(X^{**}, X^{*})$ topology. So, it only remains to show that $\pi^{-1}$ is continuous from $(X^{**}, \sigma_{G}(X^{**}, X^{*}))$ to $(X, \sigma_{G}(X, X^{*}))$. For this, fix $f \in X^{*}$. If we show that $\langle f, \pi^{-1}\xi \rangle$ is continuous on $(X^{**}, \sigma_{G}(X^{**}, X^{*}))$, then it is clear that $\pi^{-1}$ is a continuous mapping. Notice that:
    \[
    \langle f, \pi^{-1}\xi \rangle = \langle \xi, f \rangle.
    \]
    And the mapping $\xi \mapsto \langle \xi, f \rangle$ is continuous on $X^{**}$ for the $w_{G}^{*}$ topology. Hence, $B_{X}$ is $w_{G}$-compact.
        
   For the reverse implication, we know by Theorem \ref{Goldstine} that $\pi(B_{X_G}) = B_{X^{**}_{G^{**}}}$. As a consequence, $\pi(X_{G}) = X^{**}_{G^{**}}$.
\end{proof}

To conclude this section, we aim to emphasize the relationships among \(X\), \(X_G\), \(X^{**}\), \(X_{G^{**}}^{**}\), and \(\faktor{X^{**}}{G^{*}}\). To illustrate these connections, we will demonstrate that the following commutative diagram holds:

\[
\begin{tikzcd}
  X_{G} \arrow{dr}{\pi_{G}} & \\
  X \arrow{u}{S_{G}}\arrow{r}{\pi_{G} \circ S_{G}} \arrow{dr}{S_{G^{*}} \circ \pi} \arrow{d}{\pi} & X^{**}_{G^{**}} \arrow{d}{\cong} \\
  X^{**} \arrow{r}{S_{G^{*}}} & \faktor{X^{**}}{G^{*}} \\
\end{tikzcd}
\]
where $\pi$ denotes the canonical injection from $X$ into $X^{**}$, $\pi_{G}$ denotes the canonical injection of Definition \ref{Injeccio canonica G-inv}, $S_{G}$ denotes the symmetrization of the points in $X$, and $S_{G^{*}}$ denotes the symmetrization of the points in $X^{**}$.

To clarify the notation, we recall that
\[
X^{**}_{G^{*}} = \left\{f \in X^{**} ~~ | ~~ f \hbox{ is } G^{*}\hbox{-invariant}\right\},
\]\[
X^{**}_{G^{**}} = \left\{x^{**} \in X^{**} ~~ | ~~ g^{**}(x^{**}) = x^{**}\right\},
\]
where in the first case we are considering functionals that are $G^*$ invariant and in the second case we are considering points that are $G^{**}$ invariant.

Let us start by showing that the previous two spaces coincide.
\begin{Proposition}
   Let $X$ be a Banach space and $G$ a compact topological group acting on $X$. Then, $X^{**}_{G^{**}} = X^{**}_{G^{*}}$.
\end{Proposition}
\begin{proof}
    Pick $x^{**} \in X^{**}_{G^{**}}$ and $y^{*} \in X^{*}$, then:
    \[
    \langle x^{**}, g^{*}(y^{*}) \rangle = \langle g^{**}(x^{**}), y^{*} \rangle = \langle x^{**}, y^{*} \rangle.
    \]
    So, $x^{**} \in X^{**}_{G^{*}}$

    Let now $f \in X^{**}_{G^{*}}$ and $x^{*} \in X^{*}$, observe that
    \[
    \langle g^{**}(f), x^{*} \rangle = \langle f, g^{*}(x^{*}) \rangle = \langle f, x^{*} \rangle.
    \]
    So, $f \in X^{**}_{G^{**}}$.
\end{proof}

The following result establishes a fundamental relationship between $G^{*}$-invariant functionals and $G$-invariant points. Specifically, it demonstrates that any $G^{*}$-invariant functional can be derived from a $G$-invariant point, and conversely, every $G$-invariant point induces a corresponding $G^{*}$-invariant functional.

\begin{Proposition}
    Let $X$ be a Banach space and $G$ a compact topological group acting on $X$. Then, $x \in X_{G} \Leftrightarrow x^{**} \in X^{**}_{G^{*}}$
\end{Proposition}
\begin{proof}
    Take $x \in X_{G}$, observe that
    \[
    \langle x^{**}, g^{*}(y) \rangle = \langle g^{*}(y), x \rangle = \langle y, g(x) \rangle = \langle y, x \rangle = \langle x^{**}, y \rangle,
    \]
    where we have used that $x$ is $G$-invariant.

    Take now $x^{**} \in X^{**}_{G^{*}}$, it is clear that
    \[
    \langle y, g(x) \rangle = \langle x^{**}, g^{*}(y) \rangle = \langle x^{**}, y \rangle = \langle y, x \rangle \quad \forall \, y \in X^{*}.
    \]
    Applying now \cite[Lemma 5]{FaIs} we deduce that $g(x) = x$.
\end{proof}

\section{James' theorem, Krein-Smulian theorem, and Moreau's maximum formula}\label{Seccio 3}

\subsection{Hahn-Banach extension theorem}

We now aim to present a version of the Hahn-Banach extension theorem. This result complements the results presented in \cite{DFJ}. To achieve this, we first need to introduce the following well known result, which plays a critical role in the extension process. The proof of this result can be found in \cite[Lemma 3.9]{Montesinos1}.

\begin{Lemma}\label{Lema 3.9 llibre dels 6}
Let $X$ be a vectorial space, and let $f, f_{1}, \dots, f_{n}$ be linear functionals in $X$. If $\cap_{i=1}^{n}f_{i}^{-1}(0) \subseteq f^{-1}(0)$, then $f$ is a linear combination of $f_{1}, \dots, f_{n}$.
\end{Lemma}

\begin{Theorem}\label{Hahn-Banach en el predual}
Let $X$ be a Banach space, $G$ a compact topological group acting on $X$, and let $A$ be $w^{*}_{G}$-closed, convex and $G^{*}$-invariant in $X^{*}$. If $f \in X^{*} \backslash A$ is $G$-invariant, then there exists a $G$-invariant point, $x$, such that $\sup_{h \in A}\langle h, x \rangle < \langle f, x \rangle$.
\end{Theorem}
\begin{proof}
Since $A$ is $w^{*}_{G}$-closed, then, there exists $U$ a $w^{*}_{G}$-neighbourhood of the zero such that $(f + U) \cap A = \emptyset$. We can assume that $U$ is a convex neighbourhood of the zero of the form 
\[
U = \left\{y^{*} \in U^{*} ~ | ~ |\langle y^{*}, x_{i} \rangle| < \epsilon \hspace{0.2cm} \forall \, 1 \leq i \leq n\right\}
\]
for some $x_{1}, \dots, x_{n} \in X_{G}$ and $\epsilon > 0$. By simmetry of $U$ it is clear that $f \notin U + A$, and since $A + U$ is $w^{*}_{G}$-open, then $A + U$ is open and convex. Applying now Hahn-Banach separation theorem to $f$ and $A + U$, we know that there exists a $G$-invariant functional $F \in X^{**}$ such that
\[
\langle F, f \rangle > \sup_{h \in A+U}\langle F, h \rangle \geq \sup_{h \in A}\langle F, h \rangle.
\]
We claim that $F = \pi(x)$ for some $x \in X_{G}$. Fix $h_{0} \in A_{G}$ and observe that
\[
C = \sup_{U}(F) \leq \langle F, f \rangle - \langle F, h_{0} \rangle < +\infty.
\]
Consider now the points $x_{i}$ as $G$-invariant linear functionals in $X^{*}$. Let $y^{*} \in \cap_{i}x_{i}^{-1}(0)$, then $ty^{*} \in U$ for all $t > 0$. Therefore, $F(ty^{*}) \leq C$, in particular
\[
F(y^{*}) \leq \frac{C}{t} \hspace{0.2cm} \hbox{ and } \hspace{0.2cm} F(-y^{*}) \leq \frac{C}{t}.
\]
Hence $F(y^{*}) = 0$, and we just obtained that $\cap_{i}x_{i}^{-1}(0) \subseteq F^{-1}(0)$. Applying now Lemma \ref{Lema 3.9 llibre dels 6}, we deduce that $F$ is a linear combination of $x_{1}, \dots, x_{n}$ so $F \in X_{G}$ and
\[
\langle f, F \rangle > \sup_{h \in A} \langle h, F \rangle.
\]
\end{proof}

\begin{Theorem}\label{Teorema d'assolir la norma de Hahn-Banach}
Let $X$ be a Banach space and $G$ a compact topological group acting on $X$. For every $x_{0} \in X_{G}$, there exists $f_{0} \in X_{G}^{*}$ such that
\[
\Vert f_{0} \Vert = \Vert x_{0} \Vert \hspace{0.2cm} \hbox{ and } \hspace{0.2cm} \langle f_{0}, x_{0} \rangle = \Vert x_{0} \Vert^{2}
\]
\end{Theorem}
\begin{proof}
Define $H = \mathbb{R}x_{0}$, which is a $G$-invariant subspace, and
\[
\begin{array}{cccc}
  h \colon & H & \to & \mathbb{R}  \\
     & tx_{0} & \mapsto & \Vert x_{0} \Vert^{2}.
\end{array}
\]
Take $p(x) = \Vert h \Vert_{H^{*}}\Vert x \Vert$, then by the $G$-invariant version of the Hahn-Banach extension theorem (\cite[Proposition 1]{Falco}), we know that there exists a $G$-invariant functional, say $f_{0}$, such that $\Vert f_{0} \Vert_{X^{*}} = \Vert h \Vert_{H^{*}}$. Therefore:
\begin{enumerate}
    \item $\langle f_{0}, x_{0} \rangle = h(x_{0}) = \Vert x_{0} \Vert^{2}$.
    \item $\Vert f_{0} \Vert_{X^{*}} = \Vert h \Vert_{H^{*}} = \Vert x_{0} \Vert$.
\end{enumerate}
\end{proof}

\subsection{James' theorem}

Thanks to Theorem \ref{Banach-Alaoglu-Bourbaki}, we can obtain the following result that provides a key characterization of $G$-reflexivity in terms of the space of $G$-invariant points. Specifically, it asserts that the space $X$ is $G$-reflexive if, and only if, the space of $G$-invariant points, $X_G$, is reflexive. This result emphasizes the relationship between the group action and the reflexive properties of the subspace of $G$-invariant points.

\begin{Theorem}\label{Coimplicacio G-reflexivitat}
    Let $X$ be a Banach space and $G$ a compact topological group acting on $X$. Then, $X$ is $G$-reflexive if, and only if, $X_{G}$ is reflexive.
\end{Theorem}
\begin{proof}
    By Theorem \ref{Kakutani} it is enough to show that $B_X$ is $w_G(X,X^*)$-compact if, and only if, $B_{X_G}$ is  $w(X_G,X_G^*)$-compact.

    Suppose that the unit ball $B_{X_{G}}$ is $w(X_G,X_G^*)$-compact, and let us show that $B_{X}$ is $w_G(X,X^*)$-compact. For this, consider $\{V_i\}_{i\in I}$ an open covering of $B_{X}$ in $w_G(X,X^*)$. Then, for each $V_i$ we can consider the set
    \[
    U_i:=\{\overline x:x\in V_i\}.
    \]

    Note that since $V_i$ is an open set with respect to the $w_G(X,X^*)$ topology on $B_X$, we have that $U_i$ is an open set with respect to the $w(X_G,X_G^*)$ topology on $B_{X_G}$. By the $w(X_G,X_G^*)$-compacity of $B_{X_G}$ we can find a finite subcover of  $\{U_i\}_{i\in I}$ that using a abuse of notation we denote by  $\{U_i\}_{i=1}^n$. Then, the associated sets $\{V_i\}_{i=1}^n$ are a finite subcover of $B_{X}$ in $w_G(X,X^*)$. Hence $B_{X}$ is $w_G(X,X^*)$-compact.

    The other implication follows similarly.
    
\end{proof}

Recall that Falcó proved in \cite[Theorem 6]{Falco} the group invariant version of James' theorem for the space $X_{G}$. So, in view of Theorem \ref{Coimplicacio G-reflexivitat} and that result, we obtain this alternative version of $G$-invariant James' theorem.
\begin{Theorem}[James]\label{Teorema de James}
    Let $X$ be a Banach space and $G$ a compact topological group acting on $X$. Then $X$ is $G$-reflexive if, and only if, every $G$-invariant functional is norm-attaining.
\end{Theorem}

\subsection{Moreau's maximum formula}

Let us start presenting some definitions.
\begin{Definition}
    Let $(X, \Vert \cdot \Vert)$ be a Banach space and $u$ an element of $S_{X}$. We define the $G$-invariant duality mapping as follows
    \[
    J_{G}(u) = \left\{f \in X^{*}_{G} ~ | ~ \Vert f \Vert = \Vert u \Vert \hbox{ and } \langle f, u \rangle = \Vert u \Vert^{2}\right\}.
    \]
\end{Definition}

Note that Theorem \ref{Teorema d'assolir la norma de Hahn-Banach} ensures that $J_{G}(u)\ne \emptyset$ for every $u\in \S_{X_G}$.

\begin{Definition}
    Let $(X, \Vert \cdot \Vert)$ be a Banach space and $u$ an element of $S_{X_G}$. We say that a point $x \in X$ is $G$-dissipative if $\langle x^{*}, x \rangle \leq 0$ for all $x^{*} \in J_{G}(u)$. We will denote the set of all $G$-dissipative elements as follows
    \[
    Dis_{G}(X) = \left\{x \in X ~ | ~ \langle x^{*}, x \rangle \leq 0 \hspace{0.2cm} \forall \, x^{*} \in J_{G}(u) \right\}.
    \]
\end{Definition}

\begin{Definition}
    Let $(X, \Vert \cdot \Vert)$ be a Banach space, we say that $\Vert \cdot \Vert$ is $G$-strongly subdifferentiable ($G$-SSD) at $u \in S_{X}$ if the limit
    \[
    \tau(u, x) = \lim_{t \to 0^{+}}\frac{\Vert u +tx \Vert - 1}{t}
    \]
    exists uniformly for all $x \in B_{X_G}$. If this happens for all $u \in S_{X_G}$, we say that $X$ is $G$-strongly subdifferentiable.
\end{Definition}

From these definitions, we can deduce the following properties.
\begin{Proposition}
    Let $X$ be a Banach space and $G$ a compact topological group acting on $X$. Then
    \begin{enumerate}
        \item $J_{G}(u)$ is $G^{*}$-invariant, that is, $g^*(J_G(u)))=J_G(u)$ for all $g\in G$ and all $u\in S_{X_G}$.
        \item $\text{Dis}_{G}(X)$ is $G$-invariant, that is, $g(\text{Dis}_{G}(X)))=\text{Dis}_{G}(X)$ for all $g\in G$.
    \end{enumerate}
\end{Proposition}
\begin{proof}

\

    \begin{enumerate}
        \item Fix $f \in J_{G}(u)$ and $g\in G$. Then, since $f$ is $G$-invariant,
        \[
        \langle g^{*}(f), u \rangle = \langle f, g(u) \rangle = \langle f, u \rangle = \Vert u \Vert^{2}.
        \]
        Also,\[
        \Vert g^{*}f \Vert = \Vert f \Vert = \Vert u \Vert.
        \]
        And it is clear that $g^{*}\left(J_{G}(u)\right) \subseteq J_{G}(u)$ for every $g^{*} \in G^{*}$. Applying inverse mappings we obtain the other inclusion.

        \item Let $x \in \text{Dis}_{G}(X)$, it is clear that
        \[
        \langle x^{*}, g(x) \rangle = \langle x^{*}, x \rangle \leq 0 \quad \forall \, x^{*} \in J_{G}(u).
        \]
        Therefore, $g(Dis_{G}(X)) \subseteq Dis_{G}(X)$ for all $g \in G$. Applying inverse mappings again, we obtain the other inclusion.
    \end{enumerate}
\end{proof}

Let us recall the definition of the right directional derivative of an operator.

For two Banach spaces  $X, Y$, and  a subset $U \subseteq X$, the right directional derivative of $f \colon U \to Y$ at $x_{0} \in \text{int}(U)$ is defined as follows
\[
d^{+}f(x_{0})(x) = \lim_{t \to 0^{+}}\frac{f(x_{0} + tx) - f(x_{0})}{t}.
\]

Recall that for a function $f$ we have
\[
\partial f(x_0)=\{h\in X^*~~ | ~~ \langle h,x-x_0\rangle \leq f(x) -f(x_0), \, \forall \, x\in X\}.
\]
Following the same ideas than before, for a compact group $G$, and a $G$-invariant function $f$, we can define 
\[
\partial_G f(x_0)=\{h\in X_G^*~~ | ~~ \langle h,x-x_0\rangle \leq f(x) -f(x_0), \, \forall \, x\in X\}.
\]

It is obvious that if $f$ is $G$-convex and continuous
\[
\partial_G f(x_0)\subseteq \partial_G f(\overline{x_0}).
\]

 Notice also that, if $f$ is $G$-linear and continuous
\[
\partial_G f(x_0)= \partial_G f(\overline{x_0}),
\]
and if in addition $x_0$ is $G$ invariant, then
        \begin{align*}
        \partial_{G}f(x_{0}) &= \left\{x^{*} \in X^{*}_{G} ~ |~ \langle x^{*}, x - x_{0} \rangle \leq f(x) - f(x_{0} ) \hspace{0.2cm} \forall \, x \in X\right\}\\
        &= \left\{x^{*} \in X^{*}_{G} ~ |~ \langle x^{*}, \overline{x} - x_{0} \rangle \leq  f(\overline{x} ) - f(x_{0}) \hspace{0.2cm} \forall \, x \in X\right\}\\
        &= \left\{x^{*} \in X^{*}_{G} ~ |~ \langle x^{*}, x - x_{0} \rangle \leq f(x) - f(x_{0}) \hspace{0.2cm} \forall \, x \in X_G\right\}
        \end{align*}

To continue, let us provide some sufficient conditions to ensure that the right directional derivative is group invariant.
\begin{Property}
Let $X, Y$ be two Banach spaces, $U \subseteq X$ a subset, $G$ a compact topological group acting on $X$, and $f \colon U \to Y$ a mapping. Let $x_{0} \in \text{int}(U)\cap X_G$, and let $f$ be $G$-invariant, then so is $d^{+}f(x_{0})(x)$.
\end{Property}
\begin{proof}
Let $g \in G$:
\begin{align*}
d^{+}f(x_{0})(g(x)) = & \lim_{t \to 0^{+}}\frac{f(x_{0} + tg(x)) - f(x_{0})}{t} = \\ = & \lim_{t \to 0^{+}}\frac{f(g(x_{0} + tx)) - f(x_{0})}{t} = \\
 = & \lim_{t \to 0^{+}}\frac{f(x_{0} + tx) - f(x_{0})}{t} = \\ = & d^{+}f(x_{0})(x),
\end{align*}
where we have used first the $G$-invariance of $x_{0}$ and then the $G$-invariance of $f$.
\end{proof}



Now we would like to recall this well-known proposition, whose proof can be found in \cite[Proposition 1.8]{Phelps}. 
\begin{Proposition}\label{Caracteritzacio subdiferencial}
Let $X$ be a Banach space, and $f \colon X \to \mathbb{R}\cup\left\{+\infty\right\}$ be a proper and convex function. For $x_{0} \in \text{int}(\text{Dom}(f))$ we have that
\[
x^{*} \in \partial f(x_{0}) \Leftrightarrow \langle x^{*}, x \rangle \leq d^{+}f(x_{0})(x) \quad \forall \, x \in X.
\]
\end{Proposition}


A direct application of Moreau's theorem to the Banach space $X_G$ gives us the following result.
\begin{Proposition}[Moreau G-invariant]\label{Formula del maxim de Moreau}
Let $X$ be a Banach space and $G$ a compact topological group acting on $X$. Let $f \colon X \to \mathbb{R}\cup\left\{+\infty\right\}$ be a proper, convex, $G$-invariant function that is continuous at $x_{0} \in \text{Dom}(f)\cap X_G$. Then,
\[
d^{+}f(x_{0})(x) = \sup\left\{\langle x^{*}, x \rangle ~~ | ~~ x^{*} \in \partial_G f(x_{0}) \right\} \quad \forall \, x \in X_G.
\]
Moreover, this supremum is attained at some point $x^{*} \in \partial_G f(x_{0})$.
\end{Proposition}

The proof of this result is based on the fact that \[
\sup\left\{\langle x^{*}, x \rangle ~~ | ~~ x^{*} \in \partial f\vert_{X_G}(x_{0}) \right\}=\sup\left\{\langle x^{*}, x \rangle ~~ | ~~ x^{*} \in \partial_G f(x_{0}) \right\}
\]
which can be easily obtained by using the $G$-invariant Hahn-Banach extension theorem.

Note that, in general, the previous result cannot be improved in the sense that we cannot remove the condition of $x_0\in X_G$ 
as can be seen with the following example.

\begin{Example}
Let $G = \left\{Id, -Id\right\} \subseteq \mathbb{R}$, $f \colon \mathbb{R} \to \mathbb{R}$ and defined by $f(x) = |x|$. 
If $x_{0} = 0$, observe that $x_{0}$ is $G$-invariant, in fact, it is the only $G$-invariant point. On the one hand, notice that
\[
d^{+}f(0)(1) = \lim_{h \to 0^{+}} \frac{f(0+h) - f(0)}{h} = \lim_{h \to 0^{+}}\frac{|h|}{h} = 1.
\]
On the other hand, observe that
\[
\partial_{G}|\cdot|(0) = \left\{h \in \mathbb{R}_{G} ~ | ~ h(x) \leq |x|, \, \forall \, x \in \mathbb{R}\right\} = \left\{0\right\}.
\]
Thus,
\[
\sup\left\{h(x) ~ | ~ h \in \partial_{G}|\cdot|(0)\right\} = 0.
\]
Therefore, in general, the $G$-invariant Moreau's maximum formula is not true if we drop the condition of $x \in X_{G}$.
\end{Example}

\begin{Corollary}\label{Corolari formula del maxim de Moreau}
Let $X$ be a Banach space, $G$ a compact topological group acting on $X$. Then, for all $x_{0} \in X_{G}$ and for all $x \in X_G$, there exists a functional $x^{*} \in \partial_G \Vert x_{0}\Vert$ such that
\[
d^{+}\Vert x_{0}\Vert(x) = \lim_{t \to 0^{+}}\frac{\Vert x_{0} + tx \Vert - \Vert x_{0} \Vert}{t} = \langle x^{*}, x \rangle.
\]
\end{Corollary}

\begin{Corollary}
Let $X$ be a Banach space, $G$ a compact topological group acting on $X$ and $f \colon X \to Y$ be a $G$-invariant and a Gâteaux differentiable mapping. Then, for all $x_{0}, x \in X_G$ the function
\[
\begin{array}{cccl}
    \varphi \colon & [0, 1] & \to & \mathbb{R} \\
     & t & \mapsto & \Vert f(x_{0} + tx) \Vert
\end{array}
\]
admits a right derivative at every point, and there exists a $G$-invariant functional $x^{*} \in \partial\Vert \cdot \Vert (f(x_{0} + tx))$ such that
\[
\lim_{h \to 0^{+}}\frac{\varphi(t + h) - \varphi(t)}{h} = \langle x^{*}, df(x_{0} + tx)(x) \rangle.
\]
\end{Corollary}
\begin{proof}
Fix $x_{0},x \in X_G$ and $0 < t < 1$. Take $h > 0$ such that $t + h \in [0, 1]$ and observe that
\[
\frac{\varphi(t + h) - \varphi(t)}{h} = \frac{\Vert f(x_{0} + (t + h)x) \Vert - \Vert f(x_{0} + tx) \Vert}{h}.
\]
Define now the function
\[
\begin{array}{cccl}
    \phi \colon & [0, 1] & \to & Y \\
     & h & \mapsto & \frac{f(x_{0} + (t + h)x) - f(x_{0} + tx)}{h},
\end{array}
\]
which is well-defined, and by being $f$ differentiable Gâteaux at $x_{0} + tx$:
\[
\lim_{h \to 0^{+}}\phi(h) = \partial f(x_{0} + tx)(x) = \phi(0).
\]
This later meaning that, for all $\epsilon > 0$, there exists $\delta > 0$ such that
\[
\Vert \phi(h) - \phi(0) \Vert < \frac{\epsilon}{\delta} \quad \forall \, |h| < \delta,
\]
in particular
\[
\Vert f(x_{0} + tx) + h\phi(h) \Vert - \Vert f(x_{0} + tx) + h \phi(0) \Vert \leq h\Vert \phi(h) - \phi(0) \Vert < \epsilon \quad \forall \, |h| < \delta.
\]
Therefore
\begin{align*}
\lim_{h \to o^{+}}\frac{\varphi(t + h) - \varphi(t)}{h} = & \lim_{h \to 0^{+}}\frac{\Vert f(x_{0} + tx) + h\phi(h) \Vert - \Vert f(x_{0} + tx) \Vert}{h} = \\
= & \lim_{h \to 0^{+}}\frac{\Vert f(x_{0} + tx) + h\phi(0) \Vert - \Vert f(x_{0} + tx) \Vert}{h}.
\end{align*}
Applying now Corollary \ref{Corolari formula del maxim de Moreau} this later limit exists, hence $\varphi$ admits a right directional derivative in every direction, and there exists a $G$-invariant functional $x^{*} \in \partial \Vert \cdot \Vert(f(x_{0} + tx))$ such that
\[
\lim_{h \to o^{+}}\frac{\varphi(t + h) - \varphi(t)}{h} = \langle x^{*}, d f(x_{0} + tx)(x) \rangle.
\]
\end{proof}

\subsection{Krein-Smulian theorem}
\label{Krein-Smulian-section}

We aim to present a $G$-invariant version of the Krein-Smulian theorem. To this end, we will introduce the following definitions and lemmas, which lay the groundwork for the result.

\begin{Definition}
    Let $X$ be a Banach space, $G$ a compact topological group acting on $X$, and $A$ a subset of $X$. We define the $G$-invariant polar set of $A$ as follows
    \[
    A^{\circ_{G}} = \left\{x^{*} \in X^{*}_{G} ~ | ~ |\langle x^{*}, x \rangle| \leq 1, \hspace{0.2cm} \forall \, x \in A\right\}.
    \]
\end{Definition}

We have the following properties.
\begin{Proposition}\label{Polar debil estrella tancat}
    Let $X$ be a Banach space, $G$ be a compact topological group acting on $X$ and $A$ be a subset of $X$. Then, the following holds
    \begin{enumerate}
        \item $A^{\circ_{G}}$ is $G^{*}$-invariant.
        \item If $A_G\ne\emptyset$, then $A^{\circ_{G}} \subseteq (A_{G})^{\circ} \subseteq A^{\circ}$.
        
        \item If $A$ is $G$-invariant, then $A^{\circ} \subseteq (\overline{A})^{\circ}$, where $\overline{A}=\{\overline{x} ~ | ~ x\in A\}$.
        
        \item $A^{\circ_{G}}$ is $w^{*}_{G}$-close in $X^{*}$.
    \end{enumerate}
\end{Proposition}
\begin{proof}

\

    \begin{enumerate}
    \item Let $x^{*} \in A^{\circ_{G}}$, then, by $G$-invariance of $x^{*}$, for all $g \in G$, 
    \[
    |\langle g^{*}(x^{*}), x \rangle| = |\langle x^{*}, g(x) \rangle| = |\langle x^{*}, x \rangle| \leq 1.
    \]
    Thus, $g^{*}(A^{\circ_{G}}) \subseteq A^{\circ_{G}}$ for all $g^{*} \in G^{*}$. By taking inverse mappings, we obtain the other inclusion.

    \item Take $x^{*} \in A^{\circ_{G}}$, then $x^{*} \in X^{*}_{G} \subseteq X^{*}$. Moreover, $|\langle x^{*}, x \rangle| \leq 1$ for all $x \in A$, in particular for all $x \in A_{G}$. Therefore, $x^{*} \in (A_{G})^{\circ}$. The inclusion $(A_{G})^{\circ} \subseteq A^{\circ}$ follows immediately.

    \item Observe that $(\overline{A})^{\circ} = \left\{x^{*} \in X^{*} ~ | ~ |\langle x^{*}, \overline{x} \rangle| \leq 1 \hspace{0.2cm} \forall \, x \in A \right\}$. Take $x^{*} \in A^{\circ}$, then,
    \[
    |\langle x^{*}, \overline{x} \rangle| = \left|\langle x^{*}, \int_{G} g(x)d\mu(g) \rangle\right| \leq \int_{G}|\langle x^{*}, g(x) \rangle| d\mu(g) \leq 1,
    \]
    since $g(x) \in A$ for being $A$ $G$-invariant, and $x^{*} \in A^{\circ}$.

    \item We are going to show that $X^{*} \backslash A^{\circ_{G}}$ is $w^{*}_{G}$-open. Recall that for all $\phi \in X^{*}_{G}$
    \[
    \mathcal{B} = \left\{f \in X^{*} ~ | ~ \langle f - \phi, x_{i} \rangle < \epsilon \hspace{0.2cm} \forall \, 1 \leq i \leq n\right\}
    \]
    for any choices of $\epsilon > 0$, $x_{1}, \dots, x_{n} \in X_G$ is a basis of neighbourhoods of $\phi$ in the $w^{*}_{G}$ topology. Take then, $\phi \in X^{*} \backslash A^{\circ_{G}}$. By definition, we know that
    \[
    |\langle \phi, x_{0} \rangle| > 1 \quad \hbox{for some } x_{0} \in A.
    \]
    Fix $\epsilon = (\langle \phi, x_{0} \rangle - 1)/2$, define  $\mathcal{E} = \left\{f \in X^{*} ~ | ~ \langle f - \phi, x_{0} \rangle < \epsilon\right\}$. Then, for every $f\in \mathcal{E}$,
    \[
    \vert\langle f, x_{0} \rangle\vert \geq \vert\langle \phi, x_{0} \rangle \vert - \vert\langle f - \phi, x_{0} \rangle\vert  \geq\vert \langle \phi, x_{0} \rangle \vert -\epsilon > 1.
    \]
    Hence, $X^{*} \backslash A^{\circ_{G}}$ is $w^{*}_{G}$-open and the conclusion holds.
    \end{enumerate}
\end{proof}

Now we are going to show two fundamental lemmas needed to prove the Krein-Smulian theorem.
\begin{Lemma}\label{Lema del lema de Krein-Smulian}
$X$ be a Banach space and $G$ a compact topological group acting on $X$. Let $C \subseteq X^{*}$ be a $G^{*}$-invariant and a convex set such that $\delta B_{X^{*}}\cap C$ is $w^{*}_{G}$-closed for every $\delta > 0$, and, moreover, $C \cap B_{X^{*}} = \emptyset$. Then, there exists a sequence $\left\{F_{n}\right\}_{n=0}^{+\infty} \subseteq X$ of $G$-invariant finite subsets that satisfies
\begin{enumerate}[(i)]
    \item $F_{n} \subseteq \frac{1}{n}B_{X}$ for all $n \in \mathbb{N}$.
    
    \item $C \cap nB_{X^{*}} \cap F_{0}^{\circ_{G}} \cap \dots \cap F_{n-1}^{\circ_{G}} = \emptyset$ for all $n \in \mathbb{N}$.
\end{enumerate}
Moreover, there exists a sequence $\left\{x_{n}\right\}_{n=1}^{+\infty} \subseteq X_{G}$ such that:
\begin{enumerate}[(iii)]
    \item $\lim_{n \to +\infty}\Vert x_{n} \Vert = 0$,

    \item[(iv)] for all $x^{*} \in C$, there exists $n \in \mathbb{N}$ such that $|\langle x^{*}, x_{n} \rangle| > 1$.
\end{enumerate}  
\end{Lemma}
\begin{proof}
    We construct $\left\{F_{n}\right\}_{n=1}^{+\infty}$ by induction. Define $F_{0} = \left\{0\right\}$. It is clear that $(i)$ and $(ii)$ are satisfied and that $F_{0}$ is $G$-invariant. Suppose now that $F_{1}, \dots, F_{n-1}$ are all $G$-invariant, finite, and they satisfy $(i)$ and $(ii)$. Suppose, by contradiction, that
    \[
    C \cap (n+1)B_{X^{*}} \cap F_{0}^{\circ_{G}} \cap \dots \cap F_{n-1}^{\circ_{G}} \cap F^{\circ_{G}} \neq \emptyset
    \]
    for all $G$-invariant, finite subsets of $\frac{1}{n}B_{X}$, $F$. Define 
    \[
    K = C \cap (n+1)B_{X^{*}} \cap F_{0}^{\circ_{G}} \cap \dots \cap F_{n-1}^{\circ_{G}}.
    \]
    We know by hypothesis and by Proposition \ref{Polar debil estrella tancat}, $(iv)$, that $C \cap (n+1)B_{X^{*}}$ and $F_{i}^{\circ_{G}}$ are $w^{*}_{G}$-closed, and it is clear that both of them are $G$-invariant. In particular, $K$ is $w^{*}_{G}$-closed. Now, by Theorem \ref{Banach-Alaoglu-Bourbaki}, we have that $B_{X^{*}}$ is $w^{*}_{G}$-compact, and since $K$ is a $w^{*}_{G}$-closed subset of a compact subset, we deduce that $K$ is $w^{*}_{G}$-compact. Moreover, it is $G$-invariant for being a finite intersection of $G$-invariant subsets.
    
    Since $K \cap F^{\circ_{G}} \neq \emptyset$, by the finite intersection property we get that:
    \[
    I = \bigcap\left\{K \cap F^{\circ_{G}} ~ | ~ F \subseteq \frac{1}{n}B_{X} \hbox{ is any finite subset}\right\} \neq \emptyset.
    \]
    Let $x^{*} \in I$, then $x^{*} \in F^{\circ_{G}}$, so
    \[
    |\langle x^{*}, x \rangle| \leq 1 \quad \forall x \in F \subseteq \frac{1}{n}B_{X}.
    \]
    But
    \[
    |\langle x^{*}, x \rangle| \leq \Vert x^{*} \Vert \, \Vert x \Vert \leq \frac{\Vert x^{*} \Vert}{n} \leq 1.
    \]
    Hence, $\Vert x^{*} \Vert \leq n$. So, we have obtained that
    \begin{align*}
    x^{*} \in K \cap nB_{X^{*}} & = C \cap (n+1)B_{X^{*}} \cap F_{0}^{\circ_{G}} \cap \dots \cap F_{n-1}^{\circ_{G}} \cap nB_{X^{*}} \\
    & = C \cap nB_{X^{*}} \cap F_{0}^{\circ_{G}} \cap \dots \cap F_{n-1}^{\circ_{G}} = \emptyset,    
    \end{align*}
    a contradiction. Then,
    \[
    C \cap (n+1)B_{X^{*}} \cap F_{0}^{\circ_{G}} \cap \dots \cap F_{n-1}^{\circ_{G}} \cap F^{\circ_{G}} = \emptyset.
    \]
    This proves $(i)$ and $(ii)$, let us show now $(iii)$. Define $F_{n}' = \left\{\overline{x} ~ | ~ x \in F_{n}\right\}$. For the sequence $\left\{x_{n}\right\}_{n=1}^{+\infty}$ we rearrange, if required, the elements of $\cup_{n=0}^{+\infty}F_{n}'$ and call them $\left\{x_{n}\right\}_{n=1}^{+\infty}$. Observe that every $x_{n} \in \frac{1}{n}B_{X}$, this implies $(iii)$. The only thing that remains to show is that $(i)$ and $(ii)$ hold for $F_{n}'$. Define,
    \[
    K' = C \cap nB_{X^{*}} \cap (F_{0}')^{\circ_{G}} \cap \dots \cap (F_{n-1}')^{\circ_{G}}.
    \]
    We want to show that if $K'$ is non-empty, then $K$ is also non-empty. Take $x^{*} \in K'$, then it is clear, by convexity, that $\overline{x}^{*} \in nB_{X^{*}}$ and $\overline{x}^{*} \in C$. Moreover, $x^{*} \in (F_{n-1}')^{\circ}$, we want to see that $\overline{x^{*}} \in F_{n-1}^{\circ}$. This is clear by $G$-invariance of $F$:
    \[
    |\langle \overline{x}^{*}, x \rangle| = \left|\langle \int_{G^{*}}g^{*}(x^{*}) d\mu(g), x \rangle\right| = \left|\int_{G^{*}}\langle g^*(x^{*}), x \rangle d\mu(g)\right| \leq
    \]\[
    \leq \left|\int_{G}\langle x^{*}, g(x) \rangle d\mu(g)\right| = |\langle x^{*}, \overline{x} \rangle| \leq 1.
    \]
    This shows that $\overline{x}^{*} \in F_{n-1}^{\circ}$, also observe that $\overline{x}^{*} \neq 0$ since $C$ is a convex set which does not contain the zero, otherwise we would have a contradiction with the hypothesis $C \cap B_{X^{*}} = \emptyset$. And we have shown that for given $x^{*} \in K'$, then $\overline{x}^{*} \in \overline{K}$, i.e., if $K' \neq \emptyset$ then $K \neq \emptyset$. Thus, $(i)$ and $(ii)$ also hold when changing $F_{i}$ by $F_{i}'$. 
    
    Finally, let's move to the proof of $(iv)$. Let $x^{*} \in C$, there exists $n_{0} \in \mathbb{N}$ such that $\Vert x^{*} \Vert \leq n_{0}$. So, there exists an index $j \in \left\{1, \dots, n_{0}\right\}$ with $x^{*} \in F_{j}^{\circ}$. But, by $(ii)$, there exists a $k\in\{0,\ldots,j\}$ such that
    \[
    |\langle x^{*}, x_{k} \rangle| > 1.
    \]
\end{proof}

With the help of the previous lemma, we can obtain the following one.
\begin{Lemma}\label{Lema de Krein-Smulian}
Let $X$ be a Banach space, $G$ a compact topological group acting on $X$, and $C \subseteq X^{*}$ a $G^{*}$-invariant convex subset. Suppose that $\delta B_{X^{*}} \cap C$ is $w^{*}_{G}$-closed for every $\delta > 0$, and $C \cap B_{X^{*}} = \emptyset$. Then, there exists $x \in X_{G}$ such that
\[
\langle x^{*}, x \rangle \geq 1 \quad \forall \, x^{*} \in C.
\]
\end{Lemma}
\begin{proof}
By Lemma \ref{Lema del lema de Krein-Smulian} we know that there exists a sequence $\left\{x_{n}\right\}_{n=1}^{+\infty}$ of $G$-invariant points such that $\lim_{n \to +\infty}\Vert x_{n} \Vert = 0$. Then
\[
\begin{array}{cccc}
   T \colon  & X^{*} & \to & (c_{0}, \Vert \cdot \Vert_{\infty}) \\
     & x^{*} & \mapsto & \left\{\langle x^{*}, x_{n} \rangle\right\}_{n=1}^{+\infty}
\end{array}
\]
is a linear, well-defined, and continuous operator. So, since $C$ is convex, $T(C)$ is again convex. We want to show now that $C \cap nB_{X^{*}} = \emptyset$ implies that $T(C) \cap B_{c_{0}}^{\circ} = \emptyset$. Suppose, by contradiction, that there exists $z \in T(C) \cap B_{c_{0}}^{\circ}$. Then $z \in T(C)$, and there exists a point $x^{*} \in C$ such that $T(x^{*}) = z$. On the other hand, $z \in B_{c_{0}}^{\circ} = B_{l_{1}}$, therefore, $z = \left\{z_{n}\right\}_{n=1}^{+\infty}$ with $\sum_{n=1}^{+\infty}|z_{n}| < 1$. But $z_{n} = \langle x^{*}, x_{n} \rangle$, by definition of $T$, hence
\[
\sum_{n=1}^{+\infty}|\langle x^{*}, x_{n} \rangle| < 1.
\]
This is a contradiction with $(iv)$ of Lemma \ref{Lema del lema de Krein-Smulian}. Applying now the classical Hahn-Banach separation theorem to $T(C)$ and $B_{c_{0}}^{\circ}$ we know that there exists a point $y = \left\{y_{n}\right\}_{n=1}^{+\infty} \in S_{l_{1}}$ such that
\begin{equation}\label{Teo633:Eq1}
\sup\left\{\langle y, x \rangle ~ | ~ x \in B_{c_{0}}^{\circ}\right\} \leq \langle y, T(x^{*}) \rangle \quad \forall \, x^{*} \in C.
\end{equation}
It is clear that:
\begin{equation}\label{Teo633:Eq2}
\sup\left\{\langle y, x \rangle ~ | ~ x \in B_{c_{0}}^{\circ}\right\} = 1.
\end{equation}
Moreover
\[
\langle y, T(x^{*}) \rangle = \sum_{n=1}^{+\infty}y_{n} \langle x^{*}, x_{n} \rangle = \sum_{n=1}^{+\infty}\langle x^{*}, x_{n}y_{n} \rangle \quad \forall \, x^{*} \in C.
\]
By condition $(iii)$ of Lemma \ref{Lema del lema de Krein-Smulian}, we know that there exists $M > 0$ such that $\Vert x_{n} \Vert < M$ for all $n \in \mathbb{N}$. And, since $y \in S_{l_{1}}$:
\[
\Vert x_{n}y_{n} \Vert \leq \Vert x_{n} \Vert \, |y_{n}| \leq M|y_{n}|.
\]
Thus,
\[
\sum_{n=1}^{+\infty}\Vert x_{n}y_{n} \Vert \leq M\sum_{n=1}^{+\infty}|y_{n}| < +\infty.
\]
Now, since $X$ is a Banach space, every absolutely convergent series is convergent, i.e, there exists a point $x \in X$ such that
\[
x = \sum_{n=1}^{+\infty}x_{n}y_{n}.
\]
Observe that $x$ is $G$-invariant since for every $g \in G$:
\[
g(x) = g\left(\sum_{i=1}^{+\infty}x_{n}y_{n}\right) = \sum_{i=1}^{+\infty}y_{n}g(x_{n}) = \sum_{i=1}^{+\infty}x_{n}y_{n} = x,
\]
where we have applied that $x_{n}$ is $G$-invariant for every $n \in \mathbb{N}$, and the linearity and continuity of $g$. Putting this now together with \eqref{Teo633:Eq1} and \eqref{Teo633:Eq2}, we conclude the following
\[
1 \leq \langle y, T(x^{*}) \rangle = \lim_{n \to +\infty}\sum_{k=1}^{n}\langle x^{*}, x_{k}y_{k} \rangle = \langle x^{*}, x \rangle.
\]
\end{proof}

The next result will help us know when a convex subset is norm-closed.
\begin{Proposition}\label{Conjunt convex tancat}
Let $X$ be a Banach space, $G$ a compact topological group acting on $X$, and $C$ a convex subset of $ X^{*}$. If $\delta B_{X^{*}} \cap C$ is $w^{*}_{G}$-closed for every $\delta > 0$, then $C$ is norm-closed in $X^{*}$.    
\end{Proposition}
\begin{proof}
Let $\left\{x^{*}_{n}\right\}_{n=1}^{+\infty}$ be a sequence in $C$ that converges in norm to $x^{*}$ in $X^{*}$. Then, there exists a $\delta > 0$ such that $\left\{x_{n}^{*}\right\}_{n=1}^{+\infty} \subseteq \delta B_{X^{*}}$, hence, $x_{n}^{*} \in C \cap \delta B_{X^{*}}$. Since $\left\{x_{n}^{*}\right\}$ converges in norm to $x^{*}$, then it converges $w^{*}_{G}$ to $x^{*}$. Also, $C \cap \delta B_{X^{*}}$ is $w^{*}_{G}$-closed so $x^{*} \in C \cap \delta B_{X^{*}}$. In particular, $x^{*} \in C$.    
\end{proof}

Finally, we will give a $G$-invariant version of the Krein-Smulian theorem.
\begin{Theorem}[Krein-Smulian]\label{Krein-Smulian}
Let $X$ be a Banach space, $G$ a compact topological group acting on $X$, and $C \subseteq X^{*}$ a $G^*$-invariant convex set. If $\delta B_{X^{*}} \cap C$ is $w^{*}_{G}$-closed for every $\delta > 0$, then $C$ is $w^{*}_{G}$-closed in $X^{*}$.
\end{Theorem}
\begin{proof}
Observe that $C$ is $w^{*}_{G}$-closed if for all $x^{*} \in X^{*} \backslash C$ there exists a $w^{*}_{G}$-neighbourhood $U$ of $x^{*}$ such that $U \cap C = \emptyset$, i.e., $x^{*}$ is not in the $w^{*}_{G}$-closure of $C$.

Fix $x_{0}^{*} \in X^{*} \backslash C$. We are going to show now that $\overline{x_{0}^{*}} \in X^{*} \backslash C$. Suppose, by contradiction, that $\overline{x_{0}^{*}} \in C$, then for $\delta$ big enough, $x_{0}^{*}$, $\overline{x_{0}^{*}} \in \delta B_{X^{*}}$. Note that $\overline{x_{0}^{*}} \in C \cap \delta B_{X^{*}}$ but $x_{0}^{*} \notin C \cap \delta B_{X^{*}}$. Now, for $U$ a neighbourhood of $\overline{x_{0}^{*}}$ there exists $x_{1}, \ldots, x_{n} \in X_{G}$ and $\epsilon > 0$ such that, for $i=1,\ldots, n$,
\[
\langle x^{*} - \overline{x_{0}^{*}}, x_{i} \rangle \leq \epsilon.
\]
But observe that
\begin{align*}
\langle x_{0}^{*} - \overline{x_{0}^{*}}, x_{i} \rangle & = \langle x_{0}^{*}, x_{i} \rangle  -  \left\langle \int_{G}g^{*}(x_{0}^{*})d\mu(g), x_{i} \right\rangle \\ &= \langle x_{0}^{*}, x_{i} \rangle  - \int_{G}\langle x_{0}^{*}, g(x_{i}) \rangle d\mu(g) \\ &=  \langle x_{0}^{*}, x_{i} \rangle - \langle x_{0}^{*}, x_{i} \rangle = 0. 
\end{align*}
Then, $x_{0}^* \in U$, and this for every $U$ neighbourhood of $\overline{x_{0}^{*}} \in C$. Thus, $x_{0}^{*} \in \overline{C \cap \delta B_{X^{*}}}^{w_G^*} = C \cap \delta B_{X^{*}}$, since $C \cap \delta B_{X^{*}}$ is a $w^{*}_{G}$-closed set. This being a contradiction with the fact that $x_{0}^{*} \notin C$.

By Proposition \ref{Conjunt convex tancat} we know that $C$ is norm-closed, so, since $\overline{x_{0}^{*}} \in X^{*} \backslash C$, there exists a $\delta > 0$ such that $(\overline{x_{0}^{*}} + \delta B_{X^{*}}) \cap C = \emptyset$, which is equivalent to say that
\[
(\delta^{-1}(C - \overline{x^{*}})) \cap B_{X^{*}} = \emptyset.
\]
Applying now Lemma \ref{Lema de Krein-Smulian}, we know that there exists $x \in X_{G}$ such that $\langle y^{*}, x \rangle \geq 1$ for all $y^{*} \in \delta^{-1}(C - \overline{x_{0}^{*}})$. Therefore,
\[
W = \left\{x^{*} \in X^{*} ~ | ~ \langle x^{*}, x \rangle < 1 \hspace{0.2cm} \forall \, x \in X_G\right\}
\]
is a $w^{*}_{G}$-neighbourhood of zero which does not intersect $\delta^{-1}(C - \overline{x_{0}^{*}})$, i.e., $x^{*}$ is not in the $w^{*}_{G}$-closure of $C$.
\end{proof}

\section{$G$-SSD as a sufficient condition for $G$-reflexivity}\label{Seccio 4}

In this final section, we delve into an essential aspect of our investigation: establishing a necessary condition for a space to be \( G \)-reflexive. By uncovering this condition, we aim to deepen our understanding of the interplay between geometric properties of the space and the set of group invariant functionals.

\begin{Remark}
\label{remark-moreau}
By Moreau's maximum formula, Theorem \ref{Formula del maxim de Moreau}, we know that for any $x_0,x \in X_G$
\[
    \tau(x_0, x) = \lim_{t \to 0^{+}} \frac{\Vert x_0 + tx \Vert - 1}{t} = d^{+}\Vert  x_{0}\Vert(x) = \max\left\{\langle x^{*}, x \rangle ~ | ~ x^{*} \in \partial_{G}\Vert \cdot \Vert(x_0) \right\}.
\]
\end{Remark}

In general, we have \(\partial \|\cdot\| (x_0) = J(x_0)\). The following result demonstrates that this equality holds in our context as well.
    \begin{Proposition}
    Let $X$ be a Banach space, and $G$ a compact topological group acting on $X$. Then, for fixed $x_{0} \in S_{X}$, $\partial_{G}\Vert \cdot \Vert (x_{0}) = J_{G}(x_{0})$.
    \end{Proposition}
    \begin{proof}
        Recall that:
        \begin{align*}
        \partial_{G}\Vert \cdot \Vert(x_{0}) &= \left\{x^{*} \in X^{*}_{G} ~ |~ \langle x^{*}, x - x_{0} \rangle \leq \Vert x \Vert - \Vert x_{0} \Vert \hspace{0.2cm} \forall \, x \in X\right\}
        \end{align*}
        and 
        \[
        J_{G}(x_{0}) = \left\{x^{*} \in X^{*}_{G} ~ | ~ \Vert x^{*} \Vert = \Vert x_{0} \Vert, ~~ \langle x^{*}, x_{0} \rangle = \Vert x_{0} \Vert^{2}\right\}.
        \]
        Take $x^{*} \in J_{G}(x_{0})$, then applying the definition of $J_{G}(x_{0})$ and taking into account that $x_{0} \in S_{X}$, we obtain that
        \[
        \langle x^{*}, x - x_{0} \rangle = \langle x^{*}, x \rangle - \langle x^{*}, x_{0} \rangle = \langle x^{*}, x \rangle - \Vert x_{0} \Vert^{2} \leq \Vert x^{*} \Vert \, \Vert x \Vert - \Vert x_{0} \Vert^{2} =
        \]\[
        = \Vert x_{0} \Vert (\Vert x \Vert - \Vert x_{0} \Vert) = \Vert x \Vert - \Vert x_{0} \Vert.
        \]
        Then, it is clear that $J_{G}(x_{0}) \subseteq \partial_{G}\Vert \cdot \Vert(x_{0})$.

        Now take $x^{*} \in \partial_{G}\Vert \cdot \Vert(x_{0})$, by definition we know that
        \[
        \langle x^{*}, x - x_{0} \rangle \leq \Vert x \Vert - \Vert x_{0} \Vert \leq \Vert x - x_{0} \Vert.
        \]
        Hence $\Vert x^{*} \Vert \leq 1$. Now, for $x = 0$, it is clear that
        \[
        \langle x^{*}, -x_{0} \rangle \leq - \Vert x_{0} \Vert.
        \]
        Therefore
        \[
        \Vert x^{*} \Vert \geq \langle x^{*}, x_{0} \rangle \geq \Vert x_{0} \Vert = 1.
        \]
        From where we deduce that $\Vert x^{*} \Vert = 1$, and $\langle x^{*}, x_{0} \rangle = 1$. Thus, the other inclusion holds, that is, $\partial_{G}\Vert \cdot \Vert(x_{0}) \subseteq J_{G}(x_{0})$.
    \end{proof}
    As a consequence, we have that, for $x_0\in X$ and $x \in X_G$,
    \begin{align*}
     \tau(x_0, x) &\geq \tau(\overline{x_0}, x)=\max\left\{\langle x^{*}, x \rangle ~ | ~ x^{*} \in J_{G}(\overline{x_{0}})\right\}\\
    \end{align*}

Let us finish this section by providing the following properties of the dissipative set. 
\begin{Proposition}\label{Disipatiu negatiu}
    Let $(X, \Vert \cdot \Vert)$ be a Banach space, $G$ a compact topological group acting on $X$, and $u \in S_{X_G}$ be fixed. The set $\text{Dis}_{G}(X)$ has the following properties
    \begin{enumerate}[(i)]
        \item $\text{Dis}_{G}(X)$ is a non-empty closed convex cone.
        \item $\text{Dis}_{G}(X) = \left\{x \in X ~ | ~ \tau(u,\overline{x}) \leq 0\right\}$.
    \end{enumerate}
\end{Proposition}
\begin{proof}
    We are only going to show $(ii)$, the proof of $(i)$ is clear. Given $x \in \text{Dis}_{G}(X)$, by definition we know that $\langle x^{*}, x \rangle \leq 0$ for all $x^{*} \in J_{G}(u)$, in particular $\langle x^{*}, \overline{x} \rangle \leq 0$ for all $x^{*} \in J_{G}(u)$. Therefore, by Moreau's maximum formula and Remark \ref{remark-moreau}, there exists $x_{0}^{*} \in J_{G}(u)$ such that
    \[
    0 \geq \langle x_{0}^{*}, \overline{x} \rangle = \max\left\{\langle x^{*}, \overline{x} \rangle ~ | ~ x^{*} \in J_{G}(u)\right\} = \tau(u, \overline{x}).
    \]
    The other inclusion follows from the fact that for all $x^{*} \in J_{G}(u)$,
    \[
    \langle x^{*}, \overline{x} \rangle \leq \max\left\{\langle x^{*}, \overline{x} \rangle ~ | ~ x^{*} \in J_{G}(u)\right\} = \tau(u, \overline{x}) \leq 0.
    \]
\end{proof}

\begin{Proposition}\label{Interseccio dissipatiu}
    Let $X$ be a Banach  $G$-SSD space at $u \in S_{X_G}$, and let $G$ be a compact topological group acting on $X$. Then,
    \[
     B_{X} \cap \left(\text{Dis}_{G}(X)\cap X_G\right) = B_{X_G} \cap \bigcap_{t > 0}\left[-\frac{u}{t} + \left(\frac{1}{t} + \phi_{u}(t)\right)B_{X}\right],
    \]
    where
    \[
    \phi_{u}(t) \colon= \sup\left\{\frac{\Vert u + tx \Vert - 1}{t} - \tau(u, x) ~ | ~ x \in B_{X_G}\right\}.
    \]
\end{Proposition}
\begin{proof}
Given $x \in B_{X} \cap \left(\text{Dis}_{G}(X)\cap X_G\right)$, by Proposition \ref{Disipatiu negatiu} we know that $\tau(u, x) \leq 0$. Then,
\[
\frac{\Vert u + tx \Vert - 1}{t} \leq \phi_{u}(t) + \tau(u, x) \leq \phi_{u}(t) \quad \forall \, t > 0.
\]
Hence, $\Vert t^{-1}u + x \Vert \leq t^{-1} + \phi_{u}(t)$, i.e.,
\[
x \in B_{X_G} \cap \bigcap_{t > 0}\left[-\frac{u}{t} + \left(\frac{1}{t} + \phi_{u}(t)\right)B_{X}\right].
\]
On the other hand, choose $x$ in the above intersection, then:
\[
\left\Vert \frac{u}{t} + x \right\Vert \leq \frac{1}{t} + \phi_{u}(t).
\]
So
\[
\frac{\Vert u + tx \Vert - 1}{t} \leq \phi_{u}(t).
\]
Taking limits now
\[
\lim_{t \to 0^{+}}\frac{\Vert u + tx \Vert - 1}{t} \leq \lim_{t \to 0^{+}} \phi_{u}(t).
\]
But observe that $\lim_{t \to 0^{+}}\frac{\Vert u + tx \Vert - 1}{t} = \tau(u, x)$, and $\lim_{t \to 0^{+}} \phi_{u}(t) = 0$. Therefore, we have obtained that $0\geq \tau(u, x) $. This means, by Proposition \ref{Disipatiu negatiu}, that $x \in Dis_{G}(X)$.    Since by hypothesis $x\in B_{X_G}$, then $x\in B_{X} \cap \left(\text{Dis}_{G}(X)\cap X_G\right) $
\end{proof}

Let us now present two preceding results that will be crucial for establishing the main theorem of this manuscript. The first result provides a condition for determining when the dissipative nature of a set is \( w^*_G \)-closed.

\begin{Proposition}\label{Dissipatiu debil estrella tancat}
Let $X$ be a Banach space such that $X^{*}$ is $G$-SSD. Then, $ \text{Dis}_{G}(X^*)\cap X_G^*$ is $w^{*}_{G}$-closed.
\end{Proposition}
\begin{proof}
Fix $\delta > 0$ and $u^{*} \in S_{X^{*}_G}$, since $X^{*}$ is $G$-strongly subdifferentiable, by Proposition \ref{Interseccio dissipatiu}, we know that
\[
\delta B_{X^{*}} \cap \left(\text{Dis}_{G}(X^*)\cap X_G^*\right) = \delta B_{X_G^{*}} \cap \bigcap_{t > 0}\left[-\frac{u^{*}}{t} + \left(\frac{1}{t} + \phi_{u^{*}}(t)\right)B_{X^{*}}\right].
\]
Then, $\delta B_{X^{*}} \cap (\text{Dis}_{G}(X^{*}) \cap X_{G}^{*})$ is $w^{*}_{G}$-closed. Applying now Theorem \ref{Krein-Smulian} we deduce that $\text{Dis}_{G}(X^*)\cap X_G^*$ is $w^{*}_{G}$-closed.
\end{proof}

\begin{Lemma}\label{Punts en el dissipatiu}
Let $X$ be a Banach space, $G$ a compact topological group acting on $X$. Then, for $u, x\in X_G$, we have that $u- \eta x \notin \text{Dis}_{G}(X)$ if, and only, if $\eta < \tau(u, x)$.
\end{Lemma}
\begin{proof}
Observe that, for fixed $x \in S_{X_G}$, $u - \eta x \notin \text{Dis}_{G}(X)$ if, and only if, $\langle x^{*}, u - \eta x \rangle > 0$ for some $x^*\in J_G(u)$, by definition. And this condition is equivalent to show that $\langle x^{*}, u \rangle > \eta$. On the one hand, suppose that $\langle x^{*}, u \rangle > \eta$, then we know by Moreau's Maximum formula, Theorem \ref{Formula del maxim de Moreau}, that
\[
\tau(u, x) = \max\left\{\langle x^{*}, x \rangle ~ | ~ x^{*} \in J_{G}(u)\right\}.
\]
Then,
\[
\tau(u, x) \geq \langle x^{*}, u \rangle > \eta.
\]
On the other hand, suppose that $\eta < \tau(u, x)$, then there exists $x^{*} \in X^{*}$ such that $\tau(u, x) \geq \langle x^{*}, u \rangle > \eta$.
\end{proof}

We now reach the central result of this section, which captures the essence of our study and brings together the key insights developed throughout.

\begin{Theorem}
Let $X$ be a Banach space, $G$ a compact topological group acting on $X$. If $X^{*}$ is $G$-SSD, then $X$ is $G$-reflexive.
\end{Theorem}
\begin{proof}
Suppose that $\Vert \cdot \Vert^{*}$ is $G$-SSD at $u^{*} \in S_{X^{*}_G}$. We want to show that every $G$-invariant direction, say $z^{*} \in X^{*}_{G}$, is norm-attaining, so we can apply Theorem \ref{Teorema de James} and deduce that $X$ is reflexive.

We know by Proposition \ref{Dissipatiu debil estrella tancat} that $\text{Dis}_{G}(X^{*})\cap X_G^*$ is $w^{*}_{G}$-closed, convex and $G$-invariant. By Lemma \ref{Punts en el dissipatiu} we know that $u^{*} - \eta z^{*} \notin \text{Dis}_{G}(X^{*})$ if $\eta < \tau(u^{*}, z^{*})$. Applying now Theorem \ref{Hahn-Banach en el predual}, we know that there exists $x_{0} \in S_{X_G}$ such that
\begin{equation}\label{Teo639:Eq1}
\sup\left\{\langle x^{*}, x_{0} \rangle ~ | ~ x^{*} \in \text{Dis}_{G}(X^{*})\right\} < \langle u^{*}, x_{0} \rangle - \eta \langle z^{*}, x_{0} \rangle.
\end{equation}
We assert that
\[
\langle x^{*}, x_{0} \rangle \leq 0 \quad \forall \, x^{*} \in \text{Dis}_{G}(X^{*}).
\]
Otherwise, $\sup\left\{\langle x^{*}, x_{0} \rangle ~ | ~ x^{*} \in \text{Dis}_{G}(X^{*})\right\} > 0$. Hence, there exists $x_{0}^{*} \in \text{Dis}_{G}(X^{*})$ such that $\langle x_{0}^{*}, x_{0} \rangle > 0$. Since $\text{Dis}_{G}(X^{*})$ is a convex cone, this means that $nx_{0}^{*} \in \text{Dis}_{G}(X^{*})$ for every $n \in \mathbb{N}$, then
\[
\lim_{n \to +\infty}\langle nx_{0}^{*}, x_{0} \rangle = +\infty.
\]
This is a contradiction with the upper bound obtained in \eqref{Teo639:Eq1}. 

Observe that, by Theorem \ref{Teorema d'assolir la norma de Hahn-Banach}, $J_{G}(x_{0}) \neq \emptyset$. Pick $x^{*} \in J_{G}(x_{0})$. Then, again by Lemma \ref{Punts en el dissipatiu}, it is clear that
\[
x^{*} - \tau(x^{*}, z^{*})z^{*} \in \text{Dis}_{G}(X^{*}).
\]
By definition of the dissipative, we know that
\[
\langle x^{*}, x_{0} \rangle - \tau(x^*,z^*)\langle z^{*}, x_{0} \rangle \leq 0.
\]
Then
\[
\langle x^{*}, x_{0} \rangle \leq \tau(x^{*}, z^{*})\langle z^{*}, x_{0} \rangle \leq \Vert z^{*} \Vert \langle z^{*}, x_{0} \rangle \leq |\langle z^{*}, x_{0} \rangle| \leq \Vert z^{*} \Vert \, \Vert x_{0} \Vert.
\]
Notice that $\Vert z^{*} \Vert = \Vert x_{0} \Vert = \langle x^{*}, x_{0} \rangle = 1$, therefore
\[
|\langle z^{*}, x_{0} \rangle| = 1.
\]
This means that $z^{*}$ attains the norm, and the result follows.
\end{proof}

As a consequence of this, we obtain the following.
\begin{Corollary}
Let $X$ be a Banach space, $G$ a compact topological group acting on $X$. If $X^{*}$ is $G$-SSD, then every $G$-invariant functional is norm-attaining.
\end{Corollary}

We conclude the paper with a final result that establishes the existence of substantial vector spaces within the set of norm-attaining functionals. This result provides further insight into the rich geometric structure underlying the set of norm-attaining functionals in various Banach spaces.

\begin{Corollary}
    Let $X$ be a Banach space, and $G$ a compact topological group acting on $X$. If $X^{*}$ is $G$-SSD, then the set of norm-attaining operators on $X$ contains, at least, the vectorial space $X^{*}_{G}$.
\end{Corollary}

Let us conclude with the following result, which emphasizes the connection between the existence of vector spaces of norm-attaining functionals and the property of being \( G \)-reflexive for a given group \( G \). This relationship not only provides insight into the geometric structure of Banach spaces under group invariance but also allows us to obtain new results in the classical theory by leveraging \( G \)-invariant functionals. These findings enrich our understanding of \( G \)-invariant properties in functional analysis and their implications for traditional results.

\begin{Proposition}
    Let $X$ be a  Banach space. If the set of norm-attaining functionals on $X$ contains a finite dimensional Banach space $E$, then there exists a compact topological group $G$ acting on $X$ such that $E=X_G^*$. In particular, $X$ is $G$-reflexive.
\end{Proposition}
    
\begin{proof}
    Since $E$ is a finite dimensional space, we can consider $f_1,\ldots,f_n$ a base of $E$ all of them of norm one. Since the vectors are linearly independent, we can consider $x_1,\ldots,x_n$ so that they are linearly independent with $f_i(x_j)=\delta_{i,j}$ and 
    \[
X=span\{ \cap_{i=1}^n Ker(f_i), x_1,\ldots,x_n\}.  
    \]

    Consider now the operator
    $T:X \mapsto X$ given by $T(y+\sum_{i=1}^n\lambda x_i)=-y+\sum_{i=1}^n\lambda x_i$.

    Since $T^{-1}=T$ we can define the compact group $G=\{Id, T\}$. It is clear that every function $f_i$ is $G$-invariant so $E\subset X_G^*$. To see the other inclusion, consider $f\in X_G^*$. Then, if $x\in \cap_{i=1}^n Ker(f_i)$, by the definition of $T$ and using that $f$ is $G$-invariant, we have that,  
    \[
    f(T(x))=f(-x)
    \]
    hence $\cap_{i=1}^n Ker(f_i)\subset \ker(f)$. Thus, by Lemma \ref{Lema 3.9 llibre dels 6} we have that $f$ is a linear combination of $f_1,\ldots, f_n$.
   
\end{proof}

\section{Statements and declarations}
\subsection{Competing Interests}
Financial interests: Both authors of this manuscript declare they have no financial interests.

\begin{thebibliography}{999}

\bibitem{Brezis} \textsc{H. Brezis}, \textit{Functional analysis, Sobolev spaces and partial differential equations}, Springer, 2010.

\bibitem{Cohn} \textsc{D. Cohn}, \textit{Measure theory}, Birkhäuser, 2010.


\bibitem{DFJ} \textsc{S. Dantas, J. Falcó, and M. Jung},
\textit{Group invariant operators and some applications to norm-attaining theory}.
In: Results in mathematics, 78.1 (2023), pp. 1-30.





\bibitem{FaIs} \textsc{J. Falcó, and D. Isert},
\textit{Group invariant variational principles}, In: Revista de la real academia de ciencias exactas, físicas y naturales. Serie A. Matemáticas, 2024, pp.

\bibitem{Falco} \textsc{J. Falcó},
\textit{A group invariant Bishop-Phelps theorem}, In: Proceedings of the American Mathematical Society, 149.4 (2021), pp. 1609-1612


\bibitem{KadetsLopezMartin} \textsc{V. Kadets, G. López, M. Martín and D. Werner}, \textit{Equivalent norms with an extremely nonlineable set of norm attaining functionals } Journal of the Institute of Mathematics of Jussieu. 19(1) (2020), 259--279. 

\bibitem{Martin} \textsc{M. Martin},
\textit{On proximinality of subspaces and the lineability of the set of norm-attaining functionals of Banach spaces
}, Journal of Functional Analysis, 278 4 (2020), 108353.

\bibitem{Rmoutil} \textsc{M. Rmoutil},
\textit{Norm-attaining functionals need not contain 2-dimensional subspaces}, J. Funct. Anal., 272 (2017), pp. 918--928.


\bibitem{Montesinos} \textsc{M. Fabian, P. Habala, P. Hájek, V. Montesinos, and V. Zizler},
\textit{Banach space theory: the basis for linear and nonlinear analysis},
Springer, 2011.


\bibitem{Montesinos1} \textsc{M. Fabian, P. Habala, P. Hájek, V. Montesinos, J. Pelant, and V. Zizler},
\textit{Functional Analysis and Infinite-Dimensional Geometry},
Canadian Mathematical Society (Springer), 2000.





\bibitem{Phelps} \textsc{R. R. Phelps},
\textit{Convex functions, monotone operators and differentiability}.
In: Springer, 1992.


\end{thebibliography}
\end{document}